\documentclass[12pt,oneside,reqno]{amsart}
\usepackage[usenames]{color}
\usepackage[font=small,labelfont=bf]{caption}
\usepackage{sidecap}
\usepackage{float}
\usepackage{upgreek}
\usepackage{mathabx}
\usepackage{mathptmx}
\usepackage{bookman}

\usepackage{hyperref}

\usepackage{amssymb}
\usepackage{amscd}
\usepackage{array}
\usepackage{verbatim}

\theoremstyle{definition}
\newtheorem{thm}{Theorem}[section]
\newtheorem{lem}[thm]{Lemma}

\newtheorem{cnv}[thm]{Convention}

\def\T{\bold T}

\def\ccite#1{\textcolor{Red}{\cite{#1}}}
\def\amlg#1#2#3{\mathbb Z_{#1} \Asterisk_{\mathbb Z_{#3}} \mathbb Z_{#2}}
\def\free#1#2{\mathbb Z_{#1} \Asterisk \mathbb Z_{#2}}
\def\rk{\text{\bf rk}}

\numberwithin{equation}{section}

\begin{document}

\title[K-\lowercase{theory of} r\lowercase{otation} a\lowercase{lgebra} 
c\lowercase{rossed} p\lowercase{roducts}]
{\Large\rm K-\lowercase{theory of} r\lowercase{otation} a\lowercase{lgebra} 
c\lowercase{rossed} \\ p\lowercase{roducts} \lowercase{by} a\lowercase{malgamated} p\lowercase{roducts} \\ \lowercase{of} f\lowercase{inite} c\lowercase{yclic} g\lowercase{roups}}
\author{S. W\lowercase{alters}}
\date{}
\address{Department of Mathematics \& Statistics, University  of Northern B.C., Prince George, B.C. V2N 4Z9, Canada.}
\email[]{walters@unbc.ca}
\urladdr{http://hilbert.unbc.ca}

\begin{abstract}
The $K$-groups of the crossed product of the rotation C*-algebra $A_\theta$ by free and amalgamated products of the cyclic groups $\mathbb Z_n$, for $n=2,3,4,6$, are calculated. The actions here arise from the canonical actions of these groups on the rotation algebra under the flip, cubic, Fourier, and hexic automorphisms, respectively. An interesting feature in this study is that although the inclusion $A_\theta \to A_\theta \rtimes \mathbb Z_n$ induces injective maps on their $K_0$-groups, the same is not the case for the inclusions $A_\theta \rtimes \mathbb Z_d \to A_\theta \rtimes \mathbb Z_n$ for $2\le d < n \le 6$ and $d|n$, which we endeavor to calculate. Further, while for free products $K_1(A_\theta \rtimes [\free{m}{n}]) = 0$, for amalgamated products $K_1(A_\theta \rtimes [\amlg{m}nd]) = \mathbb Z^k$ is non-vanishing ($k=1,2$).
\end{abstract}

\maketitle

\begin{quote}
\begin{quote}
{\Small \tableofcontents}
\end{quote}
\end{quote}

\newpage

{\Large\section{Introduction}}
The rotation C*-algebra $A_\theta$ is the universal C*-algebra generated by two unitaries
$U,V$ satisfying the commutation relation $VU=e^{2\pi i\theta}UV$.  There are canonical actions of the finite cyclic groups $\Bbb Z_2, \Bbb Z_3, \Bbb Z_4, \Bbb Z_6$ on $A_\theta$ (where $\Bbb Z_n := \Bbb Z/n\Bbb Z$). These actions are given, respectively, by the flip $\phi$, cubic $\alpha$, Fourier $\sigma$, and hexic $\rho$, transforms:
\begin{align}
\phi(U) &= U^{-1},  \qquad \qquad \ \ \ \phi(V) = V^{-1}   \\
\alpha(U) &= e^{-\pi i\theta}U^{-1}V,  \qquad \ \alpha(V) = U^{-1} \\
\sigma(U) &= V^{-1}, \qquad\qquad\ \ \  \sigma(V) = U \\
\rho(U) &= V, \qquad\qquad\qquad \rho(V) = e^{-\pi i\theta} U^{-1}V. 
\end{align}
Of course, $\phi=\sigma^2 = \rho^3$, and $\alpha = \rho^2$. The C* and $K$-theoretic structure of these automorphism have been extensively studied in \ccite{BEEKa}, \ccite{BEEKb}, \ccite{BK}, \ccite{BW}, \ccite{ELPW}, \ccite{AK},
\ccite{SWa}, \ccite{SWb}, \ccite{SWd}, \ccite{SWe}, \ccite{SWorbifolds}.

For convenience, we shall introduce the notation
\[
\Bbb Z_{m,n} = \free{m}n, \qquad 
\Bbb Z_{m,n;d} = \amlg{m}nd
\]
for the free and amalgamated products. It is well-known that $\Bbb Z_{2,2} \cong \Bbb Z \rtimes \Bbb Z_2$ (the usual semi-direct product). It is also known that $\mathbb Z_{2,3} \cong \text{PSL}(2,\Bbb Z)$ and that $\mathbb Z_{4,6;2} \cong \text{SL}(2,\Bbb Z)$, both nonamenable groups.  (See \ccite{PH}, II.28, III.14.)
\vskip5pt

\begin{cnv}
We make free use of the notation $e(x) := e^{2\pi ix}$ for real $x$.
\end{cnv}
\vskip5pt

In joint work with Echterhoff, L\"uck, and Phillips \ccite{ELPW}, it was shown that for all real parameters $\theta$ one has 
\[
K_0(A_\theta\rtimes \Bbb Z_3) = \Bbb Z^8,\qquad 
K_0(A_\theta\rtimes \Bbb Z_4) = \Bbb Z^9,\qquad 
K_0(A_\theta\rtimes \Bbb Z_6) = \Bbb Z^{10},
\]
and that $K_1(A_\theta\rtimes \Bbb Z_n) = 0$ for $n=3,4,6$.  Before that, Kumjian \ccite{AK} already showed that for the flip case $K_0(A_\theta\rtimes \Bbb Z_2) = \Bbb Z^6$ and $K_1(A_\theta\rtimes \Bbb Z_2)=0$. 

For simplicity, we shall write
\[
K_0(A_\theta\rtimes \Bbb Z_n) = \Bbb Z^{r(n)}
\]
where $r(n)$ is the corresponding rank.  It was also shown in \ccite{ELPW} that the crossed product $A_\theta\rtimes \Bbb Z_n$ is approximately finite (AF) dimensional for $n=3,4,6$ and any irrational $\theta$. The AF result for the $n=2$ case was previously proved by Bratteli and Kishimoto \ccite{BK}.  The AF result and the $K$-groups for the Fourier case ($n=4$) was proved by the author in 2004 \ccite{SWe} for a dense $G_\delta$ set of parameters $\theta$, and in \ccite{ELPW} this is shown to hold for all irrational $\theta$.

\vskip5pt 

In this paper we prove Theorems \ref{thmfree} and \ref{thmamalg}, which give the $K$-groups of the crossed product of the rotation algebra by the canonical actions of the free products $\free{m}n$ and the amalgamated products 
\[
\amlg442, \qquad  \amlg462, \qquad  \amlg662, \qquad  \amlg663.
\]
This is done by proving Theorems \ref{Kexact} and \ref{KfreeExact} and using the results of \ccite{ELPW} and applying Natsume's exact sequence \ccite{TN}.

Since the free product groups involved here are nonnuclear (and nonamenable), with the exception of $\free22$, all the crossed products considered here, both the unreduced $C^*(A_\theta, \Bbb Z_{m,n})$ and the reduced $C_r^*(A_\theta, \Bbb Z_{m,n}) \equiv A_\theta\rtimes \Bbb Z_{m,n} $, are not nuclear since $C_r^*(\Bbb Z_{m,n})$ is not nuclear for $(m,n)\not=(2,2)$.  However, all these crossed products are K-nuclear on account of $\free{m}n$ being K-amenable and $A_\theta$ nuclear (see \ccite{BB}, 20.10.2).  Further, the K-amenability of $\Bbb Z_{m,n}$ implies that one has an isomorphism $K_0(C^*(A_\theta, \Bbb Z_{m,n})) \cong K_0(C_r^*(A_\theta, \Bbb Z_{m,n}))$, so that the K-groups obtained here are the same for the reduced and unreduced crossed products alike, thanks to a result of Cuntz \ccite{JC}.  The same applies to crossed products by the amalgamated groups $C^*(A_\theta, \Bbb Z_{m,n;d})$.

\medskip

We now state our main results.
\medskip

\begin{thm} \label{thmfree}
Under the canonical actions of the groups $\Bbb Z_n$ ($n=2,3,4,6$) of the rotation C*-algebra $A_\theta$, and for all parameters $\theta$ one has the $K_0$-groups
\begin{xalignat*}{3}
K_0(A_\theta \rtimes \Bbb Z_{2,2}) &= \Bbb Z^{12} \quad & 
K_0(A_\theta \rtimes \Bbb Z_{3,3}) &= \Bbb Z^{16}  & 
K_0(A_\theta \rtimes \Bbb Z_{4,4}) &= \Bbb Z^{18}  
\\
K_0(A_\theta \rtimes \Bbb Z_{2,3}) &= \Bbb Z^{14} \quad & 
K_0(A_\theta \rtimes \Bbb Z_{3,4}) &= \Bbb Z^{17} & 
K_0(A_\theta \rtimes \Bbb Z_{4,6}) &= \Bbb Z^{19}
\\ 
K_0(A_\theta \rtimes \Bbb Z_{2,4}) &= \Bbb Z^{15} \quad & 
K_0(A_\theta \rtimes \Bbb Z_{3,6}) &= \Bbb Z^{18} \quad & 
K_0(A_\theta \rtimes \Bbb Z_{6,6}) &= \Bbb Z^{20}
\\
K_0(A_\theta \rtimes \Bbb Z_{2,6}) &= \Bbb Z^{16}  \quad & 
 &  \quad &  & 
\end{xalignat*}
and $K_1(A_\theta \rtimes \Bbb Z_{m,n}) = 0$ in each case.
\end{thm}

\medskip

\begin{thm} \label{thmamalg}
Under the canonical actions of the groups $\Bbb Z_n$ ($n=2,3,4,6$) of the rotation C*-algebra $A_\theta$, and for all parameters $\theta$, the $K$-groups of the crossed product algebra $A_\theta \rtimes \Bbb Z_{m,n;d}$ by the amalgamated product groups, are as follows
\[
K_0( A_\theta \rtimes \Bbb Z_{4,4;2}) = \Bbb Z^{13}, \qquad
K_1( A_\theta \rtimes \Bbb Z_{4,4;2}) = \Bbb Z
\]
\[
K_0( A_\theta \rtimes \Bbb Z_{4,6;2}) = \Bbb Z^{14}, \qquad
K_1( A_\theta \rtimes \Bbb Z_{4,6;2}) = \Bbb Z,
\]
\[
K_0( A_\theta \rtimes \Bbb Z_{6,6;2}) = \Bbb Z^{16}, \qquad
K_1( A_\theta \rtimes \Bbb Z_{6,6;2}) = \Bbb Z^{2} 
\]
\[
K_0( A_\theta \rtimes \Bbb Z_{6,6;3}) = \Bbb Z^{14}, \qquad 
K_1( A_\theta \rtimes \Bbb Z_{6,6;3}) = \Bbb Z^{2}.
\]
\end{thm}
\medskip

In particular, it is interesting that unlike the free products case, the amalgamated products actions involve nonzero $K_1$'s. 

The above theorems require the following result regarding the three canonical inclusions 
\begin{align*}
\iota: A_\theta \rtimes \Bbb Z_2 &\to A_\theta \rtimes \Bbb Z_4 \\
\iota': A_\theta \rtimes \Bbb Z_2 &\to A_\theta \rtimes \Bbb Z_6 \\
\kappa: A_\theta \rtimes \Bbb Z_3 &\to A_\theta \rtimes \Bbb Z_6 
\end{align*}
and their induced $K_0$-maps, all of which are shown to be noninjective in a precise manner (in contrast with Theorem \ref{KfreeExact} below).
\begin{thm} \label{Kexact} We have the exact sequences
\[
\CD
0 @>  >> \mathbb Z @> {} >> K_0(A_\theta \rtimes \Bbb Z_2) = \mathbb Z^6 @>{\iota_*}>> 
K_0(A_\theta \rtimes \Bbb Z_4) = \mathbb Z^9 @>  >>  \mathbb Z^4 @>  >> 0
\endCD
\]
\vskip-18pt
\[
\CD
0 @>  >> \mathbb Z^2 @> {} >> K_0(A_\theta \rtimes \Bbb Z_2) = \mathbb Z^6 @>{\iota_*'}>> K_0(A_\theta \rtimes \Bbb Z_6) = \mathbb Z^{10} @>  >>  \mathbb Z^6 @>  >> 0
\endCD
\]
\[
\CD
0 @>  >> \mathbb Z^2 @> {} >> K_0(A_\theta \rtimes \Bbb Z_3) = \mathbb Z^8 @>{\kappa_*}>> K_0(A_\theta \rtimes \Bbb Z_6) = \mathbb Z^{10} @>  >>  \mathbb Z^4 @>  >> 0
\endCD
\]
where the images of $\iota_*, \iota_*', \kappa_*$ are direct summand subgroups of rank $5, 4,$ and $6$, respectively, and the free group injections into $K_0$ (in the left sides) are onto direct summands as well.
\end{thm}

\medskip

With $i: A_\theta \to A_\theta \rtimes \Bbb Z_n$ denoting the canonical inclusion, we have the following.

\begin{thm}\label{KfreeExact}
For $n=2,3,4,6$, the canonical morphism $i_*: K_0(A_\theta) \to K_0(A_\theta \rtimes \Bbb Z_n)$ maps injectively onto a direct summand. 
\end{thm}
The proofs for Theorem \ref{KfreeExact} are given at the end of Section 3 for $n = 2, 4$ (evident also from \ccite{SWa}, \ccite{SWb}, \ccite{SWd}), at the end of Section 4 for $n=6$, and at the end of Section 6 for $n=3$. 

%%%%%%%%%%%%%%%%%%%%%%%%
{\Large\section{The Natsume Exact Sequence}}

Recall from \ccite{TN} that the Natsume six-term exact sequence for crossed products by free and amalgamated actions is
\[
\CD
K_0(A \rtimes N)   
@> {i_{1*} - i_{2*}} >>  
K_0(A \rtimes G) \oplus K_0(A \rtimes H)   
@>{ j_{1*} + j_{2*} }>>  
K_0(A \rtimes (G \Asterisk_N H) )
\\
@A  AA          @.              @VV  V 
\\
K_1(A \rtimes (G \Asterisk_N H)) 
@< { j_{1*} + j_{2*} } <<  
K_1(A \rtimes G) \oplus K_1(A \rtimes H)
@< { i_{1*} - i_{2*} } <<  
K_1(A \rtimes N)
\endCD
\]
where $G,H,N$ are groups acting on the C*-algebra $A$, $N$ is a common subgroup of $G$ and $H$, and
\[
i_1: A \rtimes N \to A \rtimes G, \qquad i_2: A \rtimes N \to A \rtimes H,
\]
\[ 
j_1: A \rtimes G \to A \rtimes (G \Asterisk_N H),\qquad j_2: A \rtimes H \to A \rtimes (G \Asterisk_N H)
\]
are the canonical inclusions. (All crossed products here are reduced.)
\medskip

Applied to the cyclic groups $\Bbb Z_n, \Bbb Z_m$ acting on the rotation algebra $A_\theta$, one gets the six-term exact sequence (so $N=1$, the trivial group)
\[
\CD
\Bbb Z^2 \cong K_0(A_\theta)
@> {i_{1*} - i_{2*}} >>
K_0(A_\theta \rtimes \Bbb Z_m) \oplus K_0(A_\theta \rtimes \Bbb Z_n)
@>{ j_{1*} + j_{2*} }>>
K_0(A_\theta \rtimes \Bbb Z_{m,n} )
\\
@A  AA          @.              @VV  V
\\
K_1(A_\theta \rtimes \Bbb Z_{m,n} )
@< { j_{1*} + j_{2*} } <<
K_1(A_\theta \rtimes \Bbb Z_m) \oplus K_1(A_\theta \rtimes \Bbb Z_n)
@< { i_{1*} - i_{2*} } <<
K_1(A_\theta) \cong \Bbb Z^2
\endCD
\]
By \ccite{ELPW}, $K_1(A_\theta \rtimes \Bbb Z_m)=0$ and $K_0(A_\theta \rtimes \Bbb Z_n) \cong \Bbb Z^{r(n)}$ for all $\theta$, so the sequence becomes
\[
\CD
\Bbb Z^2 \cong K_0(A_\theta)
@> {i_{1*} - i_{2*}} >>
\Bbb Z^{r(m)} \oplus \Bbb Z^{r(n)}
@>{ j_{1*} + j_{2*} }>>
K_0(A_\theta \rtimes \Bbb Z_{m,n} )
\\
@A  AA          @.              @VV  V
\\
K_1(A_\theta \rtimes \Bbb Z_{m,n} )
@< { j_{1*} + j_{2*} } <<
\bold{0} \oplus \bold{0}
@< { i_{1*} - i_{2*} } <<
K_1(A_\theta) \cong \Bbb Z^2
\endCD
\]
Here, $i_{1*} - i_{2*}$ maps a basis of $K_0(A_\theta)$ injectively unto a direct summand of $K_0(A_\theta \rtimes \Bbb Z_m) \oplus K_0(A_\theta \rtimes \Bbb Z_n)$ in each case in view of Theorem \ref{KfreeExact} and Lemma \ref{summandlemma} below. This gives $K_1(A_\theta \rtimes \Bbb Z_{m,n})=0$. As the map $j_{1*} + j_{2*}$ has a 2-dimensional kernel, and the vertical map on the right is onto a free abelian group,  we obtain $K_0(A_\theta \rtimes \Bbb Z_{m,n}) = \Bbb Z^{r(m)+r(n)-2} \oplus \Bbb Z^2 = \Bbb Z^{r(m)+r(n)}$.  From $r(2)=6,\ r(3)=8,\ r(4)=9,\ r(6)=10$, one obtains all the $K$-groups for the free products listed in Theorem \ref{thmfree}.

\medskip

For the crossed product of $A_\theta$ by any of the amalgamated products 
$\Bbb Z_{m,n;d} = \Bbb Z_{4,4;2}$, $\Bbb Z_{4,6;2}$, $\Bbb Z_{6,6;2}$, $\Bbb Z_{6,6;3}$, where $d=2,3$, the Natsume  exact sequence is
$$
\CD
K_0(A_\theta \rtimes \Bbb Z_d)
@> {i_{1*} - i_{2*}} >>
K_0(A_\theta \rtimes \Bbb Z_m) \oplus K_0(A_\theta \rtimes \Bbb Z_n)
@>{ j_{1*} + j_{2*} }>>
K_0(A_\theta \rtimes \Bbb Z_{m,n;d} )
\\
@A  AA          @.              @VV  V
\\
K_1(A_\theta \rtimes \Bbb Z_{m,n;d})
@< { j_{1*} + j_{2*} } <<
K_1(A_\theta \rtimes \Bbb Z_m) \oplus K_1(A_\theta \rtimes \Bbb Z_n)
@< { i_{1*} - i_{2*} } <<
K_1(A_\theta \rtimes \Bbb Z_d)
\endCD
$$
\medskip

Since $K_1(A_\theta \rtimes_r \Bbb Z_n) = 0$ for $n=2,3,4,6$ (see \ccite{ELPW}), one gets
\[
\CD
\Bbb Z^{r(d)} = K_0(A_\theta \rtimes \Bbb Z_d)
@> {i_{1*} - i_{2*}} >>
\Bbb Z^{r(m)} \oplus \Bbb Z^{r(n)}
@>{ j_{1*} + j_{2*} }>>
K_0(A_\theta \rtimes \Bbb Z_{m,n;d})
\\
@A  AA          @.              @VV  V
\\
K_1(A_\theta \rtimes \Bbb Z_{m,n;d})
@< {  } <<
\bold0 \oplus \bold0
@< {  } <<
\bold0
\endCD
\]
from which one gets
\[
K_0(A_\theta \rtimes \Bbb Z_{m,n;d}) = \Bbb Z^{r(m)+r(n)-s}
\]
\[ 
K_1(A_\theta \rtimes \Bbb Z_{m,n;d}) = \ker(i_{1*} - i_{2*})
\]
where $s = r(d) - \rk(\ker(i_{1*} - i_{2*}))$ is the rank of the image of $i_{1*} - i_{2*}$, which will need to be calculated and shown to be a direct summand of $\Bbb Z^{r(m)} \oplus \Bbb Z^{r(n)}$ in all four amalgamated cases. This, together with the kernel of $i_{1*} - i_{2*}$, are calculated in this paper.

It may be worthwhile and interesting to compare our results with similar results for the group C*-algebra of $\Bbb Z_{m,n;d}$ obtained by Natsume, $K_0(C^*(\Bbb Z_{m,n;d})) = \Bbb Z^{m+n-d}$ and $K_1(C^*(\Bbb Z_{m,n;d})) = 0$ (see \ccite{TN}, Section 6). The contrasting features seem more apparent with the nonvanishing of the $K_1$-groups in the amalgamated case.
\medskip

We end this section with a lemma needed in our calculations.
\medskip

\begin{lem}\label{summandlemma}
Let $f:\mathbb Z^p \to \mathbb Z^m$ and $g:\mathbb Z^p \to \mathbb Z^n$ be group morphisms with direct summand images and such that $\ker(f) \subseteq \ker(g)$. Then the morphism 
\[
h:\mathbb Z^p \to \mathbb Z^m \oplus \mathbb Z^n, \qquad h(x) = (f(x),g(x))
\]
 has direct summand image.
\end{lem}
\begin{proof}
Let $K = \ker(f) = \ker(h)$. As $K$ is a direct summand subgroup of $\mathbb Z^p$, the quotient group $\frac{\mathbb Z^p}{K} \cong \mathbb Z^q$ for some $q$. One then gets the induced morphisms
\[
F: \mathbb Z^q = \tfrac{\mathbb Z^p}{K} \to \mathbb Z^m, \qquad F [x] = f(x)
\] 
\[
G: \mathbb Z^q = \tfrac{\mathbb Z^p}{K} \to \mathbb Z^n, \qquad G [x] = g(x)
\]
\[
H: \mathbb Z^q = \tfrac{\mathbb Z^p}{K} \to \mathbb Z^m \oplus \mathbb Z^n, 
\qquad H [x] = (f(x),g(x))
\]
in which $F,G,H$ have the same images as $f,g,h$, respectively, and $H(z) = (F(z), G(z))$ for $z \in \mathbb Z^q$. Both $H$ and $F$ are injective. (Of course, $G, H$ are well-defined by the hypothesis $\ker(f) \subseteq \ker(g)$.)
As $\ker(G)$ is also a direct summand of $\mathbb Z^q$, one can pick a basis for $\mathbb Z^q$ of the form
\[
e_1, \dots, e_r, \ e_1', \dots, e_s'
\]
where $e_1, \dots, e_r$ is a basis for $\ker(G)$. The image of $H$ has as basis the vectors
\begin{equation}\label{Himagebasis}
H(e_i) = (F(e_i),0), \qquad H(e_j') = (F(e_j'),G(e_j'))
\end{equation}
where $i=1,\dots,r$ and $j=1,\dots, s$. Our task is to show that the vectors in \eqref{Himagebasis} are contained in some basis of $\mathbb Z^m \oplus \mathbb Z^n$. 

By the hypothesis on the images of $f$ and $g$ being direct summands (and which are, respectively, the same as the images of $F$ and $G$), there is a basis of $\mathbb Z^m$ of the form
\[
F(e_1), \dots, F(e_r),\ F(e_1'), \dots, F(e_s'), \ x_1, \dots, x_\ell, 
\]
and there is a basis for $\mathbb Z^n$ of the form
\[
G(e_1'), \dots, G(e_s'), \ y_1, \dots, y_{\ell'}.
\]
From these, we obtain the following basis for $\mathbb Z^m \oplus \mathbb Z^n$:
\begin{align*}
(F(e_1),0), \dots, (F(e_r),0), &\ (F(e_1'),0), \dots, (F(e_s'),0), \ (x_1,0), \dots, (x_\ell,0)
\\
& \ (0,G(e_1')), \dots, (0,G(e_s')), \ (0,y_1), \dots, (0,y_{\ell'}).
\end{align*}
It is now evident that by adding the vectors $(0,G(e_1')), \dots, (0,G(e_s'))$, respectively, to the vectors $(F(e_1'),0), \dots, (F(e_s'),0)$ one obtains a basis for $\mathbb Z^m \oplus \mathbb Z^n$ that contains exactly the vectors in \eqref{Himagebasis} (which are a basis for the image of $H$). It follows that the image of $h$ is a direct summand.
\end{proof}

\bigskip

%%%%%%%%%%%%%%%%%%%%%%%%%%%%%%%%%%%
{\Large\section{$K$-theory of $A_\theta \rtimes \Bbb Z_{4,4;2}$}}
In this section we will calculate the morphism 
\[
\iota_*: \mathbb Z^6 = K_0(A_\theta \rtimes \Bbb Z_2) \to K_0(A_\theta \rtimes \Bbb Z_4) = \mathbb Z^9
\]
and show that its kernel is $\mathbb Z$ and its image is a direct summand subgroup isomorphic to $\mathbb Z^5$.  The canonical inclusion $\iota: A_\theta \rtimes \Bbb Z_2 \to A_\theta \rtimes \Bbb Z_4$ of crossed products is given by
\[
\iota(x + yW) =  x + yZ^2
\]
($x,y\in A_\theta$) where $W$ and $Z$ are the canonical orders 2 and 4 unitaries of the crossed products $A_\theta \rtimes \Bbb Z_2,  A_\theta \rtimes \Bbb Z_4$, respectively.

It is known from Proposition 3.2 of \ccite{SWa} that $K_0(A_\theta \rtimes \Bbb Z_2) = \mathbb Z^6$ has six basis generators $[1]=\xi_1, \xi_2, \dots, \xi_6$ with Chern-Connes characters given by the vectors
\begin{align}\label{basisT2}
\T_2(\xi_1) &= (1; 0,0,0,0), \quad \T_2(\xi_2) = (\tfrac12; 2,0,0,0), \quad 
\T_2(\xi_3) = (\tfrac12; 0,2,0,0)	\notag
\\
\T_2(\xi_4) &= (\tfrac12; 0,0,2,0), \quad \T_2(\xi_5) = (\tfrac12; 0,0,0,2), \quad 
\T_2(\xi_6) = (\tfrac\theta2; 1,c,-c,-1)
\end{align}
where $c=-1$ if $0<\theta<\frac12$ and $c=1$ if $\frac12<\theta<1$.

Recall that the Chern-Connes character for the flip automorphism $\phi$ is defined by
\[
\T_2: K_0(A_\theta \rtimes \Bbb Z_2) \to \mathbb R^5, \qquad
\T_2 = (\tau; \tau_{00}, \tau_{01}, \tau_{10}, \tau_{11})
\]
where 
\begin{equation}\label{taus}
\tau_{jk}(x + yW) = 4\phi_{jk}(y) 
\end{equation}
for $x,y\in A_\theta$, and
\begin{equation}\label{phis}
\phi_{jk}(U^mV^n) = e(-\tfrac\theta2 mn) \delta_2^{m-j} \delta_2^{n-k}
\end{equation}
where $\delta_r^s = 1$ if $r|s$, and $0$ otherwise, is the divisor delta function. Further, the canonical trace $\tau$ on the crossed product is given by $\tau(x + yW) = \tau(x)$ where $\tau(x)$ is the canonical normalized trace on the rotation algebra $A_\theta$.

We also remind ourselves, from \ccite{SWb} (see also \ccite{SWd}), that the Chern-Connes map $\T_4$ associated to the Fourier transform is 
\[
\T_4: K_0(A_\theta \rtimes \Bbb Z_4) \to \mathbb C^6, \qquad
\T_4 = (\tau; T_{10}, T_{11}; T_{20}, T_{21} T_{20})
\]
and is injective, where, for general elements $x = x_0 + x_1Z + x_2Z^2 + x_3Z^3$ in $A_\theta \rtimes \Bbb Z_4$ (where $x_j\in A_\theta$), one has
\begin{align*}
T_{10}(x) &= \psi_{10}(x_3), \qquad T_{11}(x) = \psi_{11}(x_3),
\\
T_{20}(x) &= \psi_{20}(x_2), \qquad  T_{21}(x) = \psi_{21}(x_2), \qquad T_{22}(x) = \psi_{22}(x_2)
\end{align*}
and $\psi_{ij}$ are the unbounded linear functionals (see \ccite{SWb}, Proposition 2.2)
\begin{align*}
\psi_{10}(U^mV^n) &= e(\tfrac\theta4 (m-n)^2) \updelta_2^{m-n}, \\ 
\psi_{11}(U^mV^n) &= e(\tfrac\theta4 (m-n)^2) \updelta_2^{m-n-1}
\end{align*}
and 
\begin{align*}
\psi_{20}(U^mV^n) &= \phi_{00}(U^mV^n) = e(-\tfrac\theta2 mn)  \updelta_2^{m} \updelta_2^{n}
\\
\psi_{21}(U^mV^n) &= \phi_{11}(U^mV^n) = e(-\tfrac\theta2 mn)  \updelta_2^{m-1} \updelta_2^{n-1}
\\
\psi_{22}(U^mV^n) &= (\phi_{01} + \phi_{10})(U^mV^n) = e(-\tfrac\theta2 mn) \updelta_2^{m-n-1}.
\end{align*}
The $T_{1j}$ are concentrated on the $Z^3$-component of $x$, and $T_{2j}$ are concentrated on the $Z^2$ part, and that the $T_{jk}$ are $0$ elsewhere. The canonical trace $\tau$ on the Fourier crossed product is 
\[
\tau(x_0 + x_1Z + x_2Z^2 + x_3Z^3) = \tau(x_0)
\]
where $\tau(x_0)$ is the normalized canonical trace on the rotation algebra 
$A_\theta$.

From the character table on page 645 of \ccite{SWb}, the range of $\T_4$ has the $\mathbb Z$-basis vectors
\begin{align}\label{etas}
\eta_1 &= (\tfrac12; \ 0, 0; \ \tfrac12, 0 , 0)	 \notag \\
\eta_2 &= (\tfrac12; \ \tfrac{1-i}{4}, 0; \ 0, 0 , 0)	\notag \\
\eta_3 &= (\tfrac14; \ \tfrac14, 0; \ \tfrac14, 0 , 0)	\notag \\
\eta_4 &= (\tfrac12; \ 0, 0; \ 0, \tfrac12, 0)	\notag \\
\eta_5 &= (\tfrac12; \ 0, \tfrac{1-i}{4}; \ 0, 0, 0)  \\
\eta_6 &= (\tfrac14; \ 0, \tfrac14; \ 0, \tfrac14, 0)	 \notag\\
\eta_7 &= (\tfrac\theta4; \ \tfrac{1+i}{8}, \tfrac{1+i}{8}; \ \tfrac{1}{8}, \tfrac{1}{8}, \tfrac{1}{4})	\notag \\
\eta_8 &= (\tfrac\theta4; \ \tfrac{-1+i}{8}, \tfrac{-1+i}{8}; \ -\tfrac{1}{8}, - \tfrac{1}{8}, 
- \tfrac{1}{4}) \notag	\\
\eta_9 &= (\tfrac\theta4; \ \tfrac{-1-i}{8}, \tfrac{-1-i}{8}; \ \tfrac{1}{8}, \tfrac{1}{8}, \tfrac{1}{4}). \notag
\end{align}
Since the compositions of $T_{10}$ and $T_{11}$ with the inclusion $\iota: A_\theta \rtimes \Bbb Z_2 \to A_\theta \rtimes \Bbb Z_4$ are identically 0, it follows that the induced $K_0$ map $\iota_*$ on all the basis elements $\xi_1, \dots, \xi_6$ has 0's in the $T_{10}$ and $T_{11}$ components. Further, from 
\[
T_{20} \iota_*(\xi) = \tfrac14 \tau_{00}(\xi), \quad
T_{21} \iota_*(\xi) = \tfrac14 \tau_{11}(\xi), \quad 
T_{22} \iota_*(\xi) = \tfrac14 (\tau_{01} + \tau_{10})(\xi)
\]
for any $K_0$-class $\xi$ in $K_0(A_\theta \rtimes \Bbb Z_2)$, we obtain
\[
\T_4 \iota_*(\xi) = (\tau(\xi); 0, 0; T_{20}(\xi), T_{21}(\xi), T_{22}(\xi))
= (\tau(\xi_j); 0, 0; \tfrac14\tau_{00}(\xi), \tfrac14\tau_{11}(\xi), \tfrac14(\tau_{01}+\tau_{10})(\xi)).
\]
From this we obtain the image of the generators $\xi_j$ in the span of the group generated by the $\eta_k$'s, namely 
\begin{align*}
\xi_1' := \T_4 \iota_*(\xi_1) &= (1; 0, 0; 0, 0, 0) \\
\xi_2' := \T_4 \iota_*(\xi_2) &= (\tfrac12; 0, 0; \tfrac12, 0, 0) \ = \ \eta_1 \\
\xi_3' := \T_4 \iota_*(\xi_3) &= (\tfrac12; 0, 0; 0, 0, \tfrac12) \\
\xi_4' := \T_4 \iota_*(\xi_4) &= (\tfrac12; 0, 0; 0, 0, \tfrac12) \\
\xi_5' := \T_4 \iota_*(\xi_5) &= (\tfrac12; 0, 0; 0, \tfrac12, 0) \ = \ \eta_4 \\
\xi_6' := \T_4 \iota_*(\xi_6) &= (\tfrac\theta2; 0, 0; \tfrac14, -\tfrac14, 0).
\end{align*}
Two of the vectors here are the same, $\xi_3' = \xi_4'$, giving the kernel
\[
\ker(\iota_*) = \mathbb Z(\xi_3-\xi_4),
\]
and the image of $\iota_*$ is generated by the 5 independent vectors associated to $\xi_1', \xi_2', \xi_3', \xi_5', \xi_6'$ which we need to show are a subset of a basis for the integral span of the $\eta$ vectors in \eqref{etas}. Since $\xi_2' = \eta_1$ and $\xi_5' = \eta_4$, it remains for us to express the vectors $\xi_1', \xi_3', \xi_6'$ as linear combinations of the $\eta$'s. Indeed, one checks that we have the following integral combinations
\begin{align*}
\xi_1' &= \eta_1 + \eta_2 - 2\eta_3 + \eta_4 + \eta_5 - 2\eta_6 + \eta_7 - \eta_9
\\
\xi_2' &= \eta_1
\\
\xi_3' &= \eta_1 + \eta_2 - 3\eta_3 + \eta_4 + \eta_5 - 3\eta_6 + 2 \eta_7 - \eta_8 -\eta_9
\\
\xi_5' &= \eta_4
\\
\xi_6' &= \eta_3 - \eta_4 + \eta_6 + \eta_8 + \eta_9.
\end{align*}
From these, one obtains a basis for the range of $\T_4$ containing the vectors $\{\xi_1', \xi_2', \xi_3', \xi_5', \xi_6'\}$ as follows. Using the $\xi_6'$ equation, one solves to eliminate $\eta_9$: 
\[
\eta_9 = \xi_6' - \eta_3 + \eta_4 - \eta_6 - \eta_8.
\]
Inserting this into the $\xi_1'$ equation one gets
\[
\xi_1' = \eta_1 + \eta_2 - \eta_3 + \eta_5 - \eta_6 + \eta_7 - \xi_6'  + \eta_8
\]
and into the $\xi_3'$ equation to get
\begin{equation}\label{xi3prime}
\xi_3'  = \eta_1 + \eta_2 - 2\eta_3  + \eta_5 - 2\eta_6 + 2 \eta_7  - \xi_6'.
\end{equation}
Solving the first of these for $\eta_8$ gives
\[
\eta_8 = \xi_1' + \xi_6'  - \eta_1 - \eta_2 + \eta_3 - \eta_5 + \eta_6 - \eta_7 
\]
and from \eqref{xi3prime} one could eliminate $\eta_5$ to obtain a basis consisting of the following nine vectors
\[
\eta_1 = \xi_2', \quad \eta_2, \quad \eta_3, \quad \eta_4 = \xi_5',  \quad \eta_6, \quad \eta_7, \quad \xi_1', \quad \xi_3', \quad \xi_6' 
\]
for $K_0(A_\theta \rtimes \Bbb Z_4)$ containing $\{\xi_1', \xi_2', \xi_3', \xi_5', \xi_6'\}$. This shows that the image of the canonical map
\[
\CD
 \mathbb Z^6 = K_0(A_\theta \rtimes \Bbb Z_2) @>{\iota_*}>> 
K_0(A_\theta \rtimes \Bbb Z_4) = \mathbb Z^9 
\endCD
\]
is a direct summand of $K_0(A_\theta \rtimes \Bbb Z_4)$ isomorphic to $\mathbb Z^5$, and its kernel isomorphic to $\mathbb Z$, giving the first exact sequence in Theorem \ref{Kexact}.

Therefore, in the Natsume sequence
\[
\CD
\Bbb Z^{6} = K_0(A_\theta \rtimes \Bbb Z_2)
@> {\iota_{1*} - \iota_{2*}} >>
\Bbb Z^{9} \oplus \Bbb Z^{9}
@>{ j_{1*} + j_{2*} }>>
K_0(A_\theta \rtimes_r \Bbb Z_{4,4;2})
\\
@A  AA          @.              @VV  V
\\
K_1(A_\theta \rtimes \Bbb Z_{4,4;2})
@< {  } <<
\bold0 \oplus \bold0
@< { } <<
\bold0
\endCD
\]
the map 
\[
\iota_{1*} - \iota_{2*}: K_0(A_\theta \rtimes \Bbb Z_2) \to 
K_0(A_\theta \rtimes \Bbb Z_4) \oplus K_0(A_\theta \rtimes \Bbb Z_4) = \mathbb Z^{18}
\]
has direct summand image, by Lemma \ref{summandlemma} (as $\ker \iota_{1*} = \ker \iota_{2*}$). Since its rank is $6 - \rk(\ker \iota_*) = 5$ we obtain
\[
K_0(A_\theta \rtimes \Bbb Z_{4,4;2}) = \mathbb Z^{18-5} = \mathbb Z^{13} ,
\qquad
K_1(A_\theta \rtimes \Bbb Z_{4,4;2}) = \mathbb Z
\]
as in the statement of Theorem \ref{thmamalg}. The $K_1$-group here follows since the kernel of $\iota_{1*} - \iota_{2*}$ is the same as the kernel of $\iota_*$, which is $\mathbb Z(\xi_3-\xi_4) \cong \mathbb Z$. 

\bigskip

In the remainder of this section we prove the injectivity of the map $i_*: K_0(A_\theta) \to K_0(A_\theta \rtimes \Bbb Z_2)$ as stated in Theorem \ref{KfreeExact}. Recall, $K_0(A_\theta) = \mathbb Z^2$ is generated by the classes of the identity $1$ and a Rieffel projection $e_\theta$ of trace $\theta$. For the identity, we have $\T_2 i_*[1] = (1;0,0,0,0)$ and for the Rieffel projection we have, in terms of the basis listed in \eqref{basisT2},
\[
\T_2 i_*[e_\theta] = (\theta;0,0,0,0) = 2\T_2(\xi_6) - \T_2(\xi_2) + \T_2(\xi_3) 
- \T_2(\xi_4) + \T_2(\xi_5)   
\]
so that clearly $\iota_*$ maps $K_0(A_\theta)$ to a direct summand of $K_0(A_\theta \rtimes \Bbb Z_2)$ isomorphic to $\mathbb Z^2$.
\medskip

We now show that the canonical map $i_*: K_0(A_\theta) \to K_0(A_\theta \rtimes \mathbb Z_4)$ is injective and its image is a direct summand (also part of Theorem \ref{KfreeExact}).  For the identity, we have $\T_4 i_*[1] = (1;0,0;0,0,0)$, and for the Rieffel projection we have $\T_4 i_*[e_\theta] = (\theta;0,0;0,0,0)$ (since $e_\theta$ being in $A_\theta$ has zero $Z, Z^2, Z^3$ components). Therefore, in terms of the $\eta$-basis listed in \eqref{etas},
\[
\T_4 i_*[1] = \xi_1' = \eta_1 + \eta_2 - 2\eta_3 + \eta_4 + \eta_5 - 2\eta_6 + \eta_7 - \eta_9
\]
and
\begin{align*}
\T_4 i_*[e_\theta] &= (\theta;0;0,0;0,0) = 2\xi_6' - (\xi_2' - \xi_5') 
\\
%= 4\mu_3 + 4\mu_4 + 2\mu_5 + 2\mu_6 - 2\mu_7 + 2\mu_8 - 4\mu_9 + 6\mu_{10} - \mu_2 - \mu_4 - \mu_6 + \mu_9
&= - \eta_1 + 2\eta_3 - \eta_4 + 2\eta_6 + 2\eta_8 + 2\eta_9  
\end{align*}
using the above expressions for $\xi_2', \xi_5', \xi_6'$.  The latter can be used to eliminate $\eta_1$, and from the former one eliminates $\eta_2$ to obtain the basis
\[
\T_4 i_*[1], \quad \T_4 i_*[e_\theta], \quad  
\eta_3, \quad \eta_4, \quad \eta_5, \quad \eta_6, \quad \eta_7, \quad \eta_8, \quad \eta_9 
\]
so that $i_*$ maps $K_0(A_\theta)$ onto a direct summand of $K_0(A_\theta \rtimes \Bbb Z_4)$ isomorphic to $\mathbb Z^2$. 

\bigskip

%%%%%%%%%%%%%%%%%%%%%%%%%%%%%%%%%%%
{\Large\section{$K$-theory of $A_\theta \rtimes \Bbb Z_{4,6;2}$}}

We now summarize the analogous framework for the hexic transform $\rho$. The Chern-Connes character map $\T_6$ for the crossed product $A_\theta \rtimes \Bbb Z_6$ is 
\[
\T_6: K_0(A_\theta \rtimes \Bbb Z_6) \to \mathbb C^6, \qquad
\T_6 = (\tau; H_{10}; \ H_{20}, H_{21};\ H_{30}, H_{31})
\]
(see \ccite{BW}), where, for general elements
\[
x = x_0 + x_1X + x_2X^2 + x_3X^3 + x_4X^4 + x_5X^5
\]
in $A_\theta \rtimes \Bbb Z_6$, with $X$ denoting the canonical unitary of $A_\theta \rtimes \Bbb Z_6$ for $\rho$, $X^6=I$, and $x_j\in A_\theta$, one has
\[
H_{jk}(x) = \Psi_{jk}(x_{6-j})
\]
and, in view of Theorem 3.1 \ccite{BW}, the unbounded linear functionals $\Psi_{jk}$ on $A_\theta$ are defined by
\begin{align}
\Psi_{10}(U^mV^n) &= e(\tfrac\theta2 (m^2+n^2)) \\ 
\Psi_{20}(U^mV^n) &= e(\tfrac\theta6(m-n)^2)  \delta_3^{m-n}
\\
\Psi_{21}(U^mV^n) &= e(\tfrac\theta6(m-n)^2) 
\\
\Psi_{30}(U^mV^n) &= e(-\tfrac\theta2 mn) \delta_2^{m} \delta_2^{n} 
\\
\Psi_{31}(U^mV^n) &= e(-\tfrac\theta2 mn).
\end{align}
We note that $\Psi_{30} = \phi_{00}$ and $\Psi_{31} = \phi_{00} + \phi_{01} + \phi_{10} + \phi_{11}$ (with $\phi_{jk}$ given by \eqref{phis}). Also, the map $H_{10}$ is concentrated on the $X^5$ term of $x$, $H_{2k}$ is concentrated on the $X^4$ term, $H_{3k}$ concentrated on the $X^3$, and of course, $H_{jk}$ is $0$ elsewhere. The canonical trace $\tau$ is given by
\[
\tau(x_0 + x_1X + x_2X^2 + x_3X^3 + x_4X^4 + x_5X^5) = \tau(x_0).
\]

From the character table on page 37 of \ccite{BW}, the range of $\T_6$ has $\mathbb Z$-basis
\begin{align}\label{mus}
\mu_1 &= (1; \ 0;\ 0, 0;\ 0, 0)	 \notag \\
\mu_2 &= (\tfrac16; \ \tfrac16;\ \tfrac16, \tfrac16;\ \tfrac16, \tfrac16)	\notag \\
\mu_3 &= (\tfrac16; \ \tfrac{1-\omega}6;\ -\tfrac16\omega, -\tfrac16\omega;\ -\tfrac16, -\tfrac16)	\notag \\
\mu_4 &= (\tfrac16; \ -\tfrac16\omega;\ \tfrac{\omega-1}6, \tfrac{\omega-1}6;\ \tfrac16, \tfrac16) \\
\mu_5 &= (\tfrac16; \ -\tfrac16;\ \tfrac16, \tfrac16;\ -\tfrac16, -\tfrac16)	\notag \\
\mu_6 &= (\tfrac16; \ \tfrac{\omega-1}6;\ -\tfrac16\omega, -\tfrac16\omega;\ \tfrac16, \tfrac16)	\notag \\
\mu_7 &= (\tfrac13; \ 0;\ 0, \tfrac13;\ 0, 0)	\notag \\
\mu_8 &= (\tfrac13; \ 0;\ 0, -\tfrac13\omega;\ 0, 0) \notag	\\
\mu_9 &= (\tfrac12; \ 0;\ 0, 0;\ 0, \tfrac12) \notag  \\
\mu_{10} &= (\tfrac\theta6; \ \tfrac16\omega;\ \tfrac{1+\omega}{18}, \tfrac{1+ \omega}{6};\ \tfrac1{12}, \tfrac13) \notag
\end{align}
where $\omega := e(\tfrac16) = e^{\pi i/3} = \tfrac12(1+i\sqrt3)$. 

Since the inclusion $\iota': A_\theta \rtimes \Bbb Z_2 \to A_\theta \rtimes \Bbb Z_6$ is given by $\iota'(x+yW) = x+yX^3$, its composition with $H_{10}, H_{20}, H_{21}$ are $0$. It follows that the induced $K_0$ map $\iota_*'$ on all the basis elements $\xi_1, \dots, \xi_6$ of $K_0(A_\theta \rtimes \Bbb Z_2)$ has $0$'s in the $H_{10}, H_{20}, H_{21}$ components. Further, from 
\[
H_{30} \iota' = \tfrac14 \tau_{00}, \qquad
H_{31} \iota' = \tfrac14 (\tau_{00}+\tau_{11}+\tau_{01}+\tau_{10}), \quad 
\]
(where the $\tau_{jk}$ are given by \eqref{taus}), we obtain, for any $K$-class $\xi$ in $K_0(A_\theta \rtimes \Bbb Z_2)$,
\begin{align*}
\T_6 \iota_*'(\xi) &= (\tau(\xi); 0; 0, 0; H_{30}(\xi), H_{31}(\xi))	\\
&= (\tau(\xi); 0; 0, 0;\ \tfrac14\tau_{00}(\xi), \tfrac14(\tau_{00}+ \tau_{11} + \tau_{01}+\tau_{10})(\xi)).
\end{align*}
From this we get the corresponding images of the generators $\xi_j$ in the span of the group generated by the $\mu_k$'s as follows:
\begin{align}\label{xiprimeprime}
\xi_1'' := \T_6 \iota_*'(\xi_1) &= (1; 0; 0, 0; 0, 0) \ = \ \mu_1 \notag \\
\xi_2'':= \T_6 \iota_*'(\xi_2) &= (\tfrac12; 0; 0, 0; \tfrac12, \tfrac12) \notag\\
\xi_3'' := \T_6 \iota_*'(\xi_3) &= (\tfrac12; 0; 0, 0; 0, \tfrac12) \ = \ \mu_9 \\
\xi_4'' := \T_6 \iota_*'(\xi_4) &= (\tfrac12; 0; 0, 0; 0, \tfrac12) \ = \ \mu_9 \notag\\
\xi_5'' := \T_6 \iota_*'(\xi_5) &= (\tfrac12; 0; 0, 0; 0, \tfrac12) \ = \ \mu_9 \notag\\
\xi_6'' := \T_6 \iota_*'(\xi_6) &= (\tfrac\theta2; 0; 0,0;\ \tfrac14, 0). \notag
\end{align}
Note that now {\it three} of these vectors are the same as $\mu_9$, so that they give 4 independent vectors
\[
\xi_1'' = \mu_1, \qquad \xi_2'', \qquad \xi_3'' = \mu_9, \qquad  \xi_6''
\]
spanning the image of $\T_6\iota_*'$. We want to show that these 4 vectors are contained inside some basis for the integral span of the $\mu$ vectors in \eqref{mus}. It is easy to check that 
\begin{align}\label{xi26}
\xi_2'' &= \mu_2 + \mu_4 + \mu_6
\\
\xi_6'' &= 2\mu_3 + 2\mu_4 + \mu_5 + \mu_6 - \mu_7 + \mu_8 - 2\mu_9 + 3\mu_{10}
\end{align}
from which it is now clear that the induced map 
\[
\iota_*': K_0(A_\theta \rtimes \Bbb Z_2) \to K_0(A_\theta \rtimes \Bbb Z_6) = 
\mathbb Z^{10}
\]
has kernel
\[
\ker(\iota_*') = \mathbb Z(\xi_3 - \xi_4) + \mathbb Z(\xi_4 - \xi_5) \cong \mathbb Z^2.
\]
In addition, this shows that the image of $\iota_*'$ is a direct summand isomorphic to $\mathbb Z^4$. This establishes the exact sequence in Theorem \ref{Kexact} associated to $\iota_*'$. Therefore, in the Natsume sequence 
\[
\CD
\Bbb Z^{6} = K_0(A_\theta \rtimes \Bbb Z_2)
@> {\iota_{*} - \iota_*'} >>
\Bbb Z^{9} \oplus \Bbb Z^{10}
@>{ j_{1*} + j_{2*} }>>
K_0(A_\theta \rtimes_r \Bbb Z_{4,6;2})
\\
@A  AA          @.              @VV  V
\\
K_1(A_\theta \rtimes \Bbb Z_{4,6;2})
@< {  } <<
\bold0 \oplus \bold0
@< {  } <<
\bold0
\endCD
\]
the map $\iota_{*} - \iota_*'$ has, by Lemma \ref{summandlemma}, direct summand image (the condition $\ker \iota_* \subset \ker \iota_*'$ of the lemma being met). Its rank is $6 - \rk(\ker \iota_*) = 5$ so therefore obtain
\[
K_0(A_\theta \rtimes \Bbb Z_{4,6;2}) = \mathbb Z^{9+10-5} = \mathbb Z^{14} ,
\qquad
K_1(A_\theta \rtimes \Bbb Z_{4,6;2}) = \mathbb Z
\]
as in the statement of Theorem \ref{thmamalg}, as $\ker(\iota_* - \iota_*') \cong \mathbb Z$.

\medskip

We now check that the canonical map $i_*: K_0(A_\theta) \to K_0(A_\theta \rtimes \mathbb Z_6)$ is injective and its image is a direct summand. Recall, $K_0(A_\theta) = \mathbb Z^2$ is generated by the classes of the identity $1$ and a Rieffel projection $e_\theta$ of trace $\theta$. For the identity, we have $\T_6 i_*[1] = (1;0;0,0;0,0)$ and for the Rieffel projection we have $\T_6 i_*[e_\theta] = (\theta;0;0,0;0,0)$ (since $e_\theta$ being in $A_\theta$ has zero $X^j$-components). Therefore, in terms of the $\mu$-basis listed in \eqref{mus},
\[
\T_6 i_*[1] = \mu_1  
\]
and
\begin{align*}
\T_6 i_*[e_\theta] &= (\theta;0;0,0;0,0) = 2\xi_6'' - (\xi_2'' - \mu_9) 
\\
%= 4\mu_3 + 4\mu_4 + 2\mu_5 + 2\mu_6 - 2\mu_7 + 2\mu_8 - 4\mu_9 + 6\mu_{10} - \mu_2 - \mu_4 - \mu_6 + \mu_9
&= - \mu_2 + 4\mu_3 + 3\mu_4 + 2\mu_5 + \mu_6 - 2\mu_7 + 2\mu_8 - 3\mu_9 + 6\mu_{10}
\end{align*}
first by using \eqref{xiprimeprime} and then using \eqref{xi26}.  The latter can be used to eliminate $\mu_2$ to get the basis
\[
\mu_1=\T_6 i_*[1], \quad \T_6 i_*[e_\theta], \quad  
\mu_3, \quad \mu_4, \quad \mu_5, \quad \mu_6, \quad \mu_7, \quad \mu_8, \quad \mu_9, \quad \mu_{10}, 
\]
so that $i_*$ maps $K_0(A_\theta)$ onto a direct summand of $K_0(A_\theta \rtimes \Bbb Z_6)$ isomorphic to $\mathbb Z^2$.

\bigskip

%%%%%%%%%%%%%%%%%%%%%%%%%%%%%%%%%
{\Large\section{$K$-theory of $A_\theta \rtimes \Bbb Z_{6,6;2}$}}

In the current case, again using Lemma \ref{summandlemma}, the Natsume sequence is
\[
\CD
\Bbb Z^{6} = K_0(A_\theta \rtimes \Bbb Z_2)
@> {\iota_{1*}' - \iota_{2*}'} >>
\Bbb Z^{10} \oplus \Bbb Z^{10}
@>{ j_{1*} + j_{2*} }>>
K_0(A_\theta \rtimes_r \Bbb Z_{6,6;2})
\\
@A  AA          @.              @VV  V
\\
K_1(A_\theta \rtimes \Bbb Z_{6,6;2})
@< {  } <<
\bold0 \oplus \bold0
@< {  } <<
\bold0
\endCD
\]
where the inclusion maps $\iota_{1*}', \iota_{2*}'$ into each summand were calculated in the previous section. Since the kernel of $\iota_{1*}' - \iota_{2*}'$ is the same as the kernel of $\iota_*'$, which is isomorphic to $\mathbb Z^2$, the rank of $\iota_{1*}' - \iota_{2*}'$ is 4, whence
\[
K_0(A_\theta \rtimes \Bbb Z_{6,6;2}) = \mathbb Z^{10+10-4} = \mathbb Z^{16} , \qquad
K_1(A_\theta \rtimes \Bbb Z_{6,6;2}) = \mathbb Z^2
\]
as in Theorem \ref{thmamalg}.

\bigskip

%%%%%%%%%%%%%%%%%%%%%%%%%%%%%%%%%
{\Large\section{$K$-theory of $A_\theta \rtimes \Bbb Z_{6,6;3}$}}

Here we will need to calculate the $K_0$-map
\[
\kappa_*: \mathbb Z^8 = K_0(A_\theta \rtimes \Bbb Z_3) \to K_0(A_\theta \rtimes \Bbb Z_6)
=\mathbb Z^{10}
\]
induced by the canonical inclusion $\kappa: A_\theta \rtimes \Bbb Z_3 \to 
A_\theta \rtimes \Bbb Z_6$ given by
\[
\kappa(x_0+x_1Y+x_2Y^2) = x_0+x_1X^2+x_2X^4
\]
where $Y$ is the canonical unitary of the crossed product $A_\theta \rtimes \Bbb Z_3$ (with $Y^3=I$), and $x_j\in A_\theta$. (As above, $X^6=I$.)

From Section 4 of \ccite{BW}, the Chern-Connes character map $\T_3$ takes the form 
\[
\T_3: K_0(A_\theta \rtimes \Bbb Z_3) \to \mathbb C^4, \qquad
\T_3 = (\tau; \ S_{10}, S_{11}, S_{12})
\]
where
\[
S_{1k}(x_0 + x_1Y + x_2Y^2) = \Phi_{1k}(x_{2})
\]
($k=0,1,2$) and, in view of Theorem 3.3 \ccite{BW}, the unbounded functionals $\Phi_{1k}$ on $A_\theta$ are defined by\footnote{In comparing the functionals $\Phi_{1k}$ with those in \ccite{BW}, it's important to note that we multiplied the map $\varphi_{11}$ in Theorem 3.3 of \ccite{BW} by $e(\theta/6)$ to obtain $\Phi_{11}$, and multiplied $\varphi_{12}$ by $e(4\theta/6)$ in obtaining $\Phi_{12}$ to remove the extra constant factors in $\varphi_{11}, \varphi_{12}$.}
\begin{align}
\Phi_{10}(U^mV^n) &= e(\tfrac\theta6 (m-n)^2) \delta_3^{m-n}\\ 
\Phi_{11}(U^mV^n) &= e(\tfrac\theta6 (m-n)^2) \delta_3^{m-n-1}\\ 
\Phi_{12}(U^mV^n) &= e(\tfrac\theta6 (m-n)^2) \delta_3^{m-n-2}.
\end{align}
Further, the canonical trace on the cubic crossed product $A_\theta \rtimes \Bbb Z_3$ is, as before, $\tau(x_0 + x_1Y + x_2Y^2) = \tau(x_0).$ Observe that the $\Phi$ and $\Psi$ maps are related by 
\[
\Phi_{10} = \Psi_{20}, \qquad \Phi_{10} + \Phi_{11} + \Phi_{12} = \Psi_{21}.
\]

From the character table on page 37 of \ccite{BW}, the range of $\T_3$ has the following vectors as a $\mathbb Z$-basis 
\begin{align}\label{lambdas}
\lambda_1 &= (1; \ 0, 0, 0)	 \notag \\
\lambda_2 &= (\tfrac13; \ \tfrac13, 0, 0 )	\notag \\
\lambda_3 &= (\tfrac13; \ -\tfrac\omega3, 0, 0 )	\notag \\
\lambda_4 &= (\tfrac13; \ 0, \tfrac13, 0 ) \\
\lambda_5 &= (\tfrac13; \ 0, -\tfrac\omega3, 0 )	\notag \\
\lambda_6 &= (\tfrac13; \ 0, 0, \tfrac13 )	\notag \\
\lambda_7 &= (\tfrac13; \ 0, 0, -\tfrac\omega3)	\notag \\
\lambda_8 &= (\tfrac\theta3; \ \tfrac{1+\omega}{9}, \tfrac{1+\omega}{9}, \tfrac{1+\omega}{9}   ) \notag
\end{align}
where $\omega := e(\tfrac16) = \tfrac12(1+i\sqrt3)$ as before.

Since the maps $H_{10}, H_{30}, H_{31}$ are clearly zero on the range of the inclusion $\kappa$, it follows that the induced $K_0$ map $\kappa_*$ on all the basis elements $\lambda_1, \dots, \lambda_8$ of $K_0(A_\theta \rtimes \Bbb Z_3)$ has 0's in the $H_{10}, H_{30}, H_{31}$ components. Further, from 
\[
H_{20} \kappa  = S_{10}, \qquad
H_{21} \kappa = S_{10} + S_{11} + S_{12},  
\]
we obtain (for any $K$-class $\xi$)
\begin{align*}
\T_6 \kappa_*(\xi) &= (\tau(\xi); 0; H_{20}(\kappa \xi), H_{21}(\kappa \xi); 0, 0)	\\
&= (\tau(\xi); 0;\ S_{10}(\xi), (S_{10} + S_{11} + S_{12})(\xi) ;\ 0, 0).
\end{align*}
From this we get the image of the generators $\lambda_j$ in the span of the group generated by the $\mu_k$'s as follows:
\begin{align*}
\lambda_1' := \T_6 \kappa_*(\lambda_1) &= (1; 0; 0, 0; 0, 0) \ = \ \mu_1\\
\lambda_2' := \T_6 \kappa_*(\lambda_2) &= (\tfrac13;\ 0;\ \tfrac13, \tfrac13;\ 0, 0) \ = \ \mu_2 + \mu_5 \\
\lambda_3' := \T_6 \kappa_*(\lambda_3) &= (\tfrac13; 0;\ -\tfrac\omega3, -\tfrac\omega3;\ 0, 0) \ = \ \mu_3 + \mu_6 \\
\lambda_4' := \T_6 \kappa_*(\lambda_4) &= (\tfrac13; 0;\ 0, \tfrac13;\ 0, 0) \ = \ \mu_7 \\
\lambda_5' := \T_6 \kappa_*(\lambda_5) &= (\tfrac13; 0;\ 0, -\tfrac\omega3;\ 0, 0) \ = \ \mu_8 \\
\lambda_6' := \T_6 \kappa_*(\lambda_6) &= (\tfrac13; 0;\ 0, \tfrac13;\ 0, 0) \ = \ \mu_7 \\
\lambda_7' := \T_6 \kappa_*(\lambda_7) &= (\tfrac13; 0;\ 0,-\tfrac\omega3;\ 0, 0) \ = \ \mu_8 \\
\lambda_8' := \T_6 \kappa_*(\lambda_8) &= (\tfrac\theta3; 0;\ \tfrac{1+\omega}9, \tfrac{1+\omega}3;\ 0, 0).
\end{align*}
Here we see that two pairs are the equal $\lambda_4' = \lambda_6' = \mu_7$ and $\lambda_5' = \lambda_7' = \mu_8$. Further, it can be checked that
\begin{equation}\label{lambda8}
\lambda_8'  =  \mu_3 + \mu_4 + \mu_5 - \mu_9 + 2\mu_{10}
\end{equation}
which means that the vectors
\begin{equation}\label{lambdabasis}
\lambda_1',\ \lambda_2',\ \lambda_3',\ \lambda_4', \ \lambda_5',\ \lambda_8'
\end{equation}
form a basis for the image of $\T_6\kappa_*$. (One can easily check that they are integrally independent.) Replacing $\mu_4$ using \eqref{lambda8}, $\mu_5 = \lambda_2' - \mu_2$, and $\mu_6 = \lambda_3' - \mu_3$, the 10 vectors 
\[
\mu_1 = \lambda_1', \quad \mu_2, \quad \mu_3, \quad \lambda_2', \quad \lambda_3', \quad \mu_7 = \lambda_4', \quad \mu_8 = \lambda_5', \quad \lambda_8', \quad
\mu_9, \quad \mu_{10}
\]
constitute a basis for $K_0(A_\theta \rtimes \Bbb Z_6)$ containing the basis \eqref{lambdabasis} for the image of $\T_6\kappa_*$. Therefore, it follows that the image of $\kappa_*$ is a direct summand of $K_0(A_\theta \rtimes \Bbb Z_6) = \mathbb Z^{10}$ isomorphic to $\mathbb Z^6$, and its kernel is
\[
\ker(\kappa_*) = \mathbb Z(\lambda_4-\lambda_6) + \mathbb Z(\lambda_5 -\lambda_7) \cong \mathbb Z^2.
\] 
This establishes the exact sequence in Theorem \ref{Kexact} related to $\kappa_*$. The Natsume sequence in this case then becomes
\[
\CD
\Bbb Z^{8} = K_0(A_\theta \rtimes \Bbb Z_3)
@> {\kappa_{1*} - \kappa_{2*}} >>
\Bbb Z^{10} \oplus \Bbb Z^{10}
@>{ j_{1*} + j_{2*} }>>
K_0(A_\theta \rtimes_r \Bbb Z_{6,6;3})
\\
@A  AA          @.              @VV  V
\\
K_1(A_\theta \rtimes \Bbb Z_{6,6;3})
@< {  } <<
\bold0 \oplus \bold0
@< {  } <<
\bold0
\endCD
\]
giving, again using Lemma \ref{summandlemma}, 
\[
K_0(A_\theta \rtimes \Bbb Z_{6,6;3}) = \mathbb Z^{10+10-6} = \mathbb Z^{14} , \qquad
K_1(A_\theta \rtimes \Bbb Z_{6,6;3}) = \mathbb Z^2.
\]
This proves the corresponding part of Theorem \ref{thmamalg}. 

\bigskip

We now check that the canonical map $i_*: K_0(A_\theta) \to K_0(A_\theta \rtimes \mathbb Z_3)$ is injective and its image is a direct summand. Recall, $K_0(A_\theta) = \mathbb Z^2$ is generated by the classes of the identity $1$ and a Rieffel projection $e_\theta$ of trace $\theta$. We have $\T_3 i_*[1] = (1;0,0,0)$ and for the Rieffel projection we have $\T_3 i_*[e_\theta] = (\theta;0,0,0)$ (since $e_\theta$ being in $A_\theta$ has zero $Y$ and $Y^2$ components). Therefore, in terms of the $\lambda$-basis listed in \eqref{lambdas},
\[
\T_3 i_*[1] = (1;0,0,0) = \lambda_1  
\]
and it is straightforward to verify that
\[
\T_3 i_*[e_\theta] = (\theta;0,0,0) = 3\lambda_8 - (\lambda_2 - \lambda_3) - (\lambda_4 - \lambda_5) - (\lambda_6 - \lambda_7).
\]
The latter can be used to eliminate $\lambda_2$ to get the basis
\[
\lambda_1=\T_3 i_*[1], \quad \T_3 i_*[e_\theta], \quad  
\lambda_3, \quad \lambda_4, \quad \lambda_5, \quad \lambda_6, \quad \lambda_7, \quad \lambda_8, \quad 
\]
so that $i_*$ maps $K_0(A_\theta)$ injectively onto a direct summand of $K_0(A_\theta \rtimes \Bbb Z_3)$. 
\medskip

\noindent{\bf Acknowledgements.} Thank you, Jesus!

\bigskip

%%%%%%%%%%%%%%%%%%%%%%%%%%%%%%%%%%
%%%%%%%%%%%%    REFERENCES      %%%%%%%%%%%
%%%%%%%%%%%%    REFERENCES      %%%%%%%%%%%
%%%%%%%%%%%%    REFERENCES      %%%%%%%%%%%
%%%%%%%%%%%%%%%%%%%%%%%%%%%%%%%%%%

\end{document}